\documentclass[12pt]{amsart}
\usepackage{amsmath}
\usepackage{graphicx}
\usepackage{hyperref}
\usepackage{inputenc}
\usepackage{mathtools}
\usepackage{amsfonts}
\usepackage{amssymb}
\usepackage[T1]{fontenc}
\usepackage{amsthm}
\usepackage{fullpage}
\usepackage{enumitem}
\usepackage{xcolor}

\newtheorem{theorem}{Theorem}[section]
\newtheorem{corollary}[theorem]{Corollary}
\newtheorem{lemma}[theorem]{Lemma}

\newtheorem{proposition}[theorem]{Proposition}

\theoremstyle{definition}
\newtheorem{definition}[theorem]{Definition}

\theoremstyle{remark}

\numberwithin{equation}{section}

\newcommand{\F}{\mathbb{F}_q}
\newcommand{\Fm}{\mathbb{F}_{q^m}}

\newcommand{\N}{\mathfrak{N}}

\newcommand{\C}{\mathcal{C}_{\nu}}

\DeclareMathOperator{\Tr}{Tr}
\DeclareMathOperator{\lcm}{lcm}
\DeclareMathOperator{\Ord}{Ord}

\title{Normal and primitive normal elements with prescribed traces in intermediate extensions of finite fields}

\keywords{Finite fields; Primitive elements; Normal elements; Additive and multiplicative characters; Trace}
\subjclass[2020]{12E20, 11T24}

\author{Arpan Chandra Mazumder}
\address{Department of Mathematical Sciences, Tezpur University, Tezpur, Assam, 784028, India}
\email{arpan10@tezu.ernet.in}

\author{Giorgos Kapetanakis}
\address{Department of Mathematics, University of Thessaly, 3rd km Old National Road Lamia-Athens, 35100, Lamia, Greece}
\email{kapetanakis@uth.gr}

\author{Dhiren Kumar Basnet}
\address{Department of Mathematical Sciences, Tezpur University, Tezpur, Assam, 784028, India}
\email{dbasnet@tezu.ernet.in}

\thanks{The first author is supported by DST INSPIRE Fellowship, Govt. of India (IF210206).}

\begin{document}
	\begin{abstract}
	 In this article, we study the existence and distribution of elements in finite field extensions with prescribed traces in several intermediate extensions that are also either normal or primitive normal. In the former case, we fully characterize the conditions under which such elements exist and provide an explicit enumeration of these elements. In the latter case we provide asymptotic results.
	\end{abstract}
	
	\maketitle
	
\section{Introduction}

Let $q$ be a prime power and $\F$ the finite field of order $q$. For any given positive integer $m$, let $\Fm$ denote the extension field of $\F$ of degree $m$. The multiplicative group $\Fm^*$ is cyclic and a generator of this group is called a \emph{primitive} element of $\Fm$. An element $\alpha \in \Fm$ is said to be \emph{normal over $\F$} (or just \emph{normal} if the choice of the base field is clear) if the set of all its conjugates with respect to $\F$, that is, if the set $\{\alpha, \alpha^{q}, \ldots , \alpha^{q^{m-1}}\}$  forms a basis of $\Fm$ over $\F$. An element $\alpha \in \Fm$ is said to be \emph{primitive normal} if it is both primitive and normal over $\F$.

The motivation behind the study of primitive and normal elements derives from both theoretical and practical matters.	
Namely, primitive elements, besides their theoretical interest, have various applications, including cryptographic schemes \cite{difhell} such as the Diffie-Hellman key exchange, the ElGamal Encryption scheme and the construction of Costas arrays \cite{costas}, which are also used in sonar and radar technology. Normal elements hold computational advantages for finite field arithmetic and are therefore used in many software and hardware implementations, most notably, in coding theory and cryptography. 

Another property that has attracted interest is prescribing the trace of an element $\alpha \in \Fm$ over $\F$. The \emph{trace} of $\alpha\in\Fm$ over $\F$
is the sum of all conjugates of $\alpha$ with respect to $\F$, that is, $\Tr_{\mathbb{F}_{q^m}/\mathbb{F}_{q}}(\alpha)= \alpha+\alpha^{q}+\ldots+\alpha^{q^{m-1}}$.
For the sake of simplicity, since in this work we are dealing with intermediate extensions of $\Fm/\F$, from now on, for $m>1$, $d\mid m$ and $\alpha \in \Fm$, we denote the trace of $\alpha$ over $\mathbb{F}_{q^d}$ by
\begin{equation*}
	\Tr_{{m}/{d}}(\alpha)=\sum_{i=0}^{m/d-1} \alpha^{q^{id}}.
\end{equation*}

In this line of work, in 1990, Cohen~\cite{cohen1990} established the existence of primitive elements with a prescribed trace up to some genuine exceptions. 

\begin{theorem}[{\cite[Theorem~1.1]{cohen1990}}] \label{pmtrace}
	Let $q$ be a prime power, $m$ a positive integer and $a \in \F$. Then there
	exists a primitive element $\alpha \in \Fm$ such that $\Tr_{m/1}(\alpha) = a$ unless $a = 0$ and $m = 2$ or
	$a = 0, m = 3$ and $q = 4$.
\end{theorem}

Subsequently, in 1999, Morgan and Mullen's conjecture \cite{morgan1994} was proven by Cohen and Hachenberger~\cite{cohen1999}, where they established the existence of a primitive normal element with nonzero prescribed trace. Observe that a normal element never has trace equal to zero, whence the assumption that the trace is nonzero is necessary.

Recently, Reis~\cite{reis}, characterized the existence of a solution for a special family of linear equations over finite fields and determined the exact number of solutions. As an application, Reis and Ribas~\cite{ribas} studied the existence and distribution of primitive elements in intermediate extensions of finite fields.

As a natural continuation of the aforementioned works, in this paper, we explore the existence of normal and primitive normal elements in $\Fm$ with prescribed traces in several intermediate extensions $\mathbb{F}_{q^d}$ of $\Fm$. In particular, for given $m>1$, $d_1<d_2<\ldots<d_k$ divisors of $m$, and $a_j \in \mathbb{F}_{q^{d_j}}$, we discuss the existence of a normal and of a primitive normal element $\alpha \in \Fm$ such that, for each $1 \leq j \leq k$,
\[
\Tr_{{n}/{d_j}}(\alpha)=\sum_{i=0}^{n/d_j-1} \alpha^{q^{id_j}}=a_j.
\]

In particular, not only we fully characterize the 
necessary conditions for the case of normal elements with prescribed intermediate traces, but we also explicitly enumerate them, see Theorem~\ref{thm:normality}.

In addition, regarding primitive normal elements with prescribed intermediate traces, we obtain asymptotic and concrete results under the restriction $\gcd(d_i, d_j)=1$ for $1 \leq i < j \leq k$, that are displayed in Theorem~\ref{th5.1}.

The paper is structured as follows. In Section~\ref{sec:prelim}, we introduce some useful notation and background material. Section~\ref{sec:normality} is devoted in studying the necessary conditions and the explicit enumeration of normal elements with prescribed traces in several intermediate extensions. In Section~\ref{sec:main_result}, we obtain an asymptotic condition for the existence of desired primitive normal elements in $\Fm$ with prescribed traces in several intermediate extensions. Finally, in Section~\ref{sec:existence_results}, we obtain some concrete existence results.

\section{Preliminaries} \label{sec:prelim}
In this section, we recall some definitions and results and provide some preliminary notations which are used to prove the main results of this article.

\subsection{Linearized polynomials and $\F$-order}

Before we proceed further, we mention some essential facts on linearized polynomials that we will use along the way. For more details on this important family of polynomials over finite fields, we refer the interested readers to \cite[Section~3.4]{lidlniederreiter97}.

\begin{definition}
	A polynomial $L_f\in\F[x]$ of the form
	\[ L_f(x) = \sum_{i=0}^k f_i x^{q^i} \]
	is called a \emph{linearized} polynomial. Moreover, if $f = \sum_{i=0}^k f_ix^i\in\F[x]$, then the $L_f$ above is the \emph{$q$-associate} of $f$.
\end{definition}

The following properties of linearized polynomials are well-known and straightforward. We refer the interested readers to \cite{handbook} and the references therein for more details. 

\begin{proposition} \label{prop:linearized}
	Let $f,g\in \F[x]$ be two polynomials and let $L_f$ and $L_g$ be their $q$-associates. Then, for every $a,b\in\F$,
	\begin{enumerate}
		\item $L_f(ax+by) = a L_f(x) + b L_f(y)$ and
		\item $L_f(L_g(x)) = L_{fg}(x)$.
	\end{enumerate}
\end{proposition}

\begin{definition}
	The \emph{$\F$-order} of some $\beta\in\Fm$, denoted by $\Ord_q(\beta)$ is the minimum degree monic polynomial over $\F$, such that $L_{\Ord_q(\beta)}(\beta) = 0$.
\end{definition}

Within the literature, the $\F$-order is commonly referred to as the \emph{additive} order as a nod to the fact that the additive group $\Fm$ can be viewed as an $\F[x]$-module. Next, observe that $L_{x^m-1}(\beta) = 0$ for all $\beta\in\Fm$, i.e., the $\F$-order of an element of $\Fm$ exists and is of degree at most $m$. In fact, the following results hold, while their proofs are straightforward.

\begin{proposition} \label{propo:F-order}
	Let $\beta\in\Fm$. The following are true:
	\begin{enumerate}
		\item $\Ord_q(\beta) \mid x^m-1$.
		\item $\beta$ is normal over $\F$ if and only if $\Ord_q(\beta) = x^m-1$.
		\item If $d\mid m$, then $\beta\in\mathbb F_{q^d}$ if and only if $\Ord_q(\beta) \mid x^d-1$.
		\item If $f\in\F[x]$, then $\Ord_q(L_f(\beta)) = \Ord_q(\beta) / \gcd(f,\Ord_q(\beta))$.
	\end{enumerate}
\end{proposition}

In a similar fashion, the \emph{$\F$-order} of an additive character $\psi$ of $\Fm$ is denoted by $\Ord_q(\psi)$ and is defined as the minimum degree monic polynomial over $\F$, such that $\psi\left( L_{\Ord_q(\psi)}(\beta) \right) = 1$,
for all $\beta\in\Fm$. Furthermore, Proposition~\ref{propo:F-order} entails that for all additive characters $\psi$ of $\Fm$, $\Ord_q(\psi)\mid x^m-1$.

\subsection{Characteristic functions}

Fix a positive integer $m$, $d$ a divisor of $m$ and $a \in \mathbb{F}_{q^d}$. Let $\rho_m$ be the characteristic function for primitive elements in $\Fm$, and $\kappa_m$ be the characteristic function for normal elements in $\Fm$ over $\mathbb{F}_{q^d}$.
In particular it is well-known that for any $\beta \in \Fm$,
\[ 
\rho_m(\beta)=\theta(q) \sum_{t\mid q^m-1} \left( \frac{\mu(t)}{\phi(t)} \sum_{\eta \in \Gamma(t)} \eta(\beta) \right),
\] 
where $\theta(q):= \phi(q^n-1)/(q^n-1)$, $\mu$ is the M\"obius function
and $\Gamma(t)$ stands for the set of multiplicative characters of order $t$. 
Likewise, for any $\beta\in\Fm$,
\[ 
\kappa_m(\beta)= \Theta(x^m-1) \sum_{f\mid x^m-1} \left( \frac{\mu^\prime(f)}{\Phi(f)} \sum_{\psi \in \Gamma(f)} \psi(\beta) \right),
\] 
where $\Theta(x^m-1):= \Phi(x^m-1)/{q^{m}}$, $\Phi$ is the analogue of the Euler $\phi$ function defined as
\[ \Phi(f) = \left|\left( \frac{\F[x]}{\langle f\rangle} \right)^*\right| , \] $\Gamma(f)$ stands for the set of additive characters of $\F$-order $f$ and $\mu^\prime$ is the analogue of the M\"obius function defined as
\[
\mu^\prime(g)=\begin{cases}
	(-1)^s , & \text{if $g$ is the product of $s$ distinct irreducible monic polynomials}, \\
	0 , &\text{otherwise.}\\ 
\end{cases}
\]

\subsection{The trace map}

Let $n$ be a divisor of $m$ and $\gamma \in \Fm$ be such that $\Tr_{{\mathbb{F}_{q^m}}/{{\mathbb{F}_{q^n}}}}(\gamma) = a \in \mathbb{F}_{q^n}$. Let $\chi$ denote the canonical additive character of $\Fm$, then all the additive characters of $\Fm$ are given by $\chi_c$, where $\chi_c(\alpha)=\chi(c\alpha)$ for any $c\in\Fm$ and $\alpha\in\Fm$. For any $\beta \in \Fm$, if $\tau_{m,d,a}$ stands for the characteristic function for elements in $\Fm$ with trace $a$ over $\mathbb{F}_{q^d}$, then
\begin{equation*}
	\tau_{m,n,a}(\beta)=\frac{1}{q^n}\sum_{c \in \mathbb{F}_{q^n}} \chi_c(\beta-\gamma)=\frac{1}{q^n}\sum_{c \in \mathbb{F}_{q^n}} \chi_c(\beta)\chi_c(\gamma)^{-1}.
\end{equation*}

The trace is transitive, that is, if $e$ divides $d$ and $d$ divides $n$, then for any $\alpha \in \Fm$ we have that $\Tr_{m/e}(\alpha)= \Tr_{d/e}(\Tr_{m/d}(\alpha))$. In particular, if $d_1<\ldots<d_k$ are divisors of $m$ and we choose $a_i \in \mathbb{F}_{q^{d_i}}$, $1\leq i \leq k$, then the existence of an element $\alpha \in \Fm$	with $\Tr_{m/d_i}(\alpha)=a_i$
is necessarily conditional on the following identities:
\begin{equation}\label{cond 2.1}
	\Tr_{d_i/\gcd(d_i,d_j)}(a_i)=\Tr_{m/\gcd(d_i,d_j)}(\alpha)=\Tr_{d_j/\gcd(d_i,d_j)}(a_j), \quad 1\leq i,j \leq k .
\end{equation}

Recently, Reis~\cite[Theorem~4.1]{reis} showed that Eq.~\eqref{cond 2.1} is also sufficient and that there exist exactly $q^{m-\lambda(\mathbf d)}$ elements in $\Fm$ with $\Tr_{m/d_i}(\alpha)=a_i$ for $1 \leq i \leq k$, where 
\begin{align*}
	\lambda(\mathbf d) & =  \deg(\lcm(x^{d_1}-1,\ldots,x^{d_k}-1))\\
	& =  d_1+\cdots+d_k+\sum_{i=2}^{k}(-1)^{i+1}\sum_{1\leq l_1<\ldots<l_i \leq k}^{} \gcd(d_{l_1},\ldots,d_{l_i}).
\end{align*}

Eq.~\eqref{cond 2.1} implies that if $d_i\mid d_j$, then 
$\Tr_{m/d_i}(\alpha)=a_i$ is already implied by $\Tr_{m/d_j}(\alpha)=a_j$. Therefore, without loss of generality, we may restrict ourselves to the divisors $d_1< \ldots < d_k$ of $n$ such that $d_i \nmid d_j$ for any $1 \leq i < j \leq k$. Next, we introduce the following, which we adopt from \cite{reis}.

\begin{definition}\label{def2.1}
	Let $m$ be an integer and $1<k<\sigma_0(m)$ , where $\sigma_0(m)$ denotes the number of positive divisors of $m$.
	\begin{enumerate}[label=(\roman*)]
		\item $\lambda_k(m)$ stands for the set of $k$-tuples $\mathbf{d}=(d_1,\ldots, d_k)$, where $d_1<\ldots<d_k<m$ are divisors of $m$ such that $d_i\nmid d_j$ for every $1\leq i < j \leq k$.
		\item For $\mathbf{d}=(d_1,\ldots,d_k) \in \lambda_k(m)$, set $\mathbb{F}_{\mathbf{d}}=\prod_{i=1}^{k} \mathbb{F}_{q^{d_i}}$ and 
		\[ \lambda(\mathbf{d}) = d_1+\cdots+d_k+\sum_{i=2}^{k}(-1)^{i+1}\sum_{1\leq l_1<\ldots<l_i \leq k}^{} \gcd(d_{l_1},\ldots,d_{l_i}). \]
	\end{enumerate}
	Moreover, for $\mathbf{d}=(d_1,\ldots,d_k) \in \lambda_k(m)$ and $\mathbf{a}=(a_1, \ldots, a_k) \in \mathbb{F}_{\textbf{d}}$, the $k$-tuple $\mathbf{a}$ is \emph{$\mathbf{d}$-admissible} if, for any $1\leq i < j\leq k$, 
	\[ \Tr_{d_i/\gcd(d_i,d_j)}(a_i)=\Tr_{d_j/\gcd(d_i,d_j)}(a_j) . \]
\end{definition}

\subsection{Some estimates}

Finally, we will need the following in establishing our main result.

\begin{lemma}[{\cite[Corollary~1.2]{reis}}]\label{l2.3}
	Let $m>1$ be an integer, $1<k<\sigma_0(m)$ and let $\mathbf{d} = (d_1, \ldots , d_k) \in \lambda_k(m)$. Then the number of $k$-tuples $(x_1, \ldots , x_k) \in \mathbb{F}_{\mathbf{d}}$ such that
	$x_1+\cdots+x_k=0$ equals \[ q^{d_1+\cdots+d_k-\lambda(\mathbf{d})}. \]
\end{lemma}

For each $n \in \mathbb{N}$, we denote by $\omega(n)$ and $W(n)$, the number prime divisors of $n$ and the number of square-free divisors of $n$ respectively. Also for $f(x) \in \F[x]$, we denote by $\omega(f)$ and $W(f)$, the number of monic irreducible $\F$-divisors of $f$ and the number of square-free $\F$-divisors of $f$ respectively. The following results provide bounds on $W(q^m-1)$ and $W(x^m-1)$, respectively.

\begin{lemma}[{\cite[Lemma~3.7]{hucz2}}]\label{l2.4}
	For any $\alpha \in \mathbb{N}$ and a positive real number $\nu$, $W(\alpha) \leq \C\cdot \alpha^{1/\nu}$, where $\C = \prod_{i=1}^{r} \frac{2}{p_i^{1/\nu}}$ and $p_1, p_2, \dots, p_r$ are the primes less than or equal to $2^\nu$ that divide $\alpha$. In particular, we will require the following values of $\C$ in the computations ahead
	\begin{enumerate}[label=(\roman*)]
		\item $\mathcal{C}_{11}= 4.2445\cdot 10^{14}$
		\item $\mathcal{C}_{12}= 1.0573\cdot 10^{24}$ and
		\item $\mathcal{C}_{31}=2.4015 \cdot 10^{1553069}$.
	\end{enumerate}
\end{lemma}

\begin{lemma}[{\cite[Lemma~2.9]{lens}}]\label{l2.5}
	Let $q$ be a prime power and $m$ a positive integer. Then, we have $W(x^m-1) \leq 2^{\frac{1}{2}(m+\gcd(m,q-1))}$. In particular, $W(x^m-1) \leq 2^m$, while the equality holds if and only if $m \mid (q-1)$. Furthermore, if $m \nmid (q-1)$, $W(x^m-1) \leq  2^{3m/4}$ since in this case, $\gcd(m, q-1) \leq \frac{m}{2}$.
\end{lemma}

The following is a direct consequence of \cite[Ineq.~(4.1)]{cohen2015consecutive}.

\begin{lemma}\label{l2.6}
	Let $W(t)$ denote the number of squarefree divisors of $t$. Then for $t\geq 3$, 
	\begin{equation*}
		W(t-1) < t^{0.96/\log\log t} .
	\end{equation*}
\end{lemma}

\section{Intermediate Traces of Normal Elements} \label{sec:normality}

In this section we study the existence of normal elements of $\Fm$ over $\F$, with their traces over several intermediate extensions arbitrarily prescribed. Throughout this section, $m$ is relatively prime to $q$, $1<k<\sigma_0(m)$, $\mathbf{d}= (d_1, \ldots , d_k) \in \lambda_k(m)$ and $\mathbf{a}=(a_1, \ldots, a_k) \in \mathbb{F}_{\mathbf{d}}$ is a $\mathbf{d}$-admissible $k$-tuple.

\begin{lemma} \label{lemma:normal_trace}
	Suppose $\beta\in\Fm$ is normal over $\F$ and $d\mid m$. Then $\Tr_{m/d}(\beta)$ is normal over $\F$ (as an element of $\mathbb F_{q^d}$).
\end{lemma}
\begin{proof}
	Set $\Tr_{m/d}(\beta)=b$. Then $b = L_{\frac{x^m-1}{x^d-1}}(\beta)$. Assume that $b\in\mathbb F_{q^d}$ is not normal over $\F$. Then $\deg(\Ord_q(b)) < d$. Moreover
	\[ L_{\Ord_q(b)}(b) = 0 \Rightarrow L_{\Ord_q(b)} \left( L_{\frac{x^m-1}{x^d-1}}(\beta) \right) = 0 \Rightarrow L_{\Ord_q(b)\frac{x^m-1}{x^d-1}} (\beta) = 0 . \]
	The latter contradicts the normality of $\beta$, since $\deg\left( \Ord_q(b)\frac{x^m-1}{x^d-1} \right) < m$.
\end{proof}

The above implies that we cannot arbitrarily prescribe the trace of a normal element over intermediate extensions, but instead we have to confine ourselves to values of the corresponding trace functions that are, themselves, normal over the base field. This renders the following definition essential for our setting.

\begin{definition}
	Some $\mathbf{d}$-admissible $\mathbf{a}=(a_1, \ldots, a_k) \in \mathbb{F}_{\mathbf{d}}$ is \emph{normal} if $a_i\in\mathbb F_{q^{d_i}}$ is normal over $\F$ for every $i=1,\ldots ,k$.
\end{definition}

Next, we focus on the inverse problem and obtain a correspondence, via the trace map, between the elements of $\mathbb F_{q^d}$ that are normal over $\F$ and the elements of $\Fm$ that are normal over $\F$, where $d\mid m$.
Towards this end, we continue with the following auxiliary lemma.

\begin{lemma} \label{lem:aux_polys}
	Let $f,g\in\F[x]$ be polynomials such that $f\mid g$. The map
	\[ \xi : \left( \frac{\F[x]}{\langle g \rangle}\right)^* \to \left( \frac{\F[x]}{\langle f \rangle}\right)^* , \ h\pmod{g} \mapsto h\pmod{f} \]
	is a group epimorphism.
\end{lemma}
\begin{proof}
	The only nontrivial part of this claim is that $\xi$ is onto. Write, $g=f g'g''$, where we take $g'\in\F[x]$ to be the largest degree divisor of $g$ that is relatively prime to $f$ and $g'' = g/(g'f)$. It follows that $\Phi(g) = \Phi(f) \Phi(g') q^{\deg g''}$ and, given that the domain and the co-domain of $\xi$ have orders $\Phi(g)$ and $\Phi(f)$, respectively, $\xi$ is onto if and only if $|\ker\xi| = \Phi(g')q^{\deg g''}$.
	
	Now, take some $h\in\F[x]$ of degree less than $\deg(g)$, such that $h+\langle g\rangle\in\ker\xi$. Then $h = 1+fk$, for some $k\in\F[x]$ of degree less than $\deg(g)-\deg(f)$, while $\gcd(h,g) = 1$. This means that, out of the $q^{\deg(g)-\deg(f)}$ choices of $k$, we are left with those such that
	\[ 1+fk \not\equiv \ell\pmod{g'} \iff k\not\equiv (\ell-1)f^{-1} \pmod{g'} , \]
	for all $\ell\in\F[x]$ of degree less than $\deg(g')$ that are not relatively prime to $g'$. In other words, we are left with $\Phi(g')$ distinct choices for $k$ modulo $g'$. By comparing degrees, we readily obtain that each such choice corresponds to $q^{\deg g''}$ choices of degree at most $\deg g-\deg f$. Hence, $|\ker\xi|=\Phi(g')q^{\deg g''}$.
\end{proof}

\begin{theorem} \label{thm:trace_correspondence}
	Let $m$ and $d$ be such that $d\mid m$. The mapping
	\[ \nu : \{\gamma\in\Fm : \gamma \text{ normal over }\F \} \to \{c\in\mathbb F_{q^d} : c \text{ normal over }\F \} , \ \gamma \mapsto \Tr_{m/d}(\gamma) \]
	is a $k$-to-one correspondence, where $k = \Phi(x^m-1)/\Phi(x^d-1)$.
\end{theorem}
\begin{proof}
	Lemma~\ref{lemma:normal_trace} implies that $\nu$ is well-defined. Next, fix some normal $\beta\in\Fm$. Proposition~\ref{propo:F-order} implies that every normal element $\gamma$ of $\Fm$ can be written as $\gamma = L_h(\beta)$ for some $h\in\F[x]$, that is relatively prime to $x^m-1$ and is unique modulo $x^m-1$. In other words, there is a correspondence between the normal elements of $\Fm$ and the group $(\F[x] / \langle x^m-1 \rangle)^*$. In a similar fashion the normal elements of $\mathbb F_{q^d}$ correspond to the group  $(\F[x] / \langle x^d-1 \rangle)^*$.
	The desired result follows from Lemma~\ref{lem:aux_polys} upon observing that the trace of $\gamma$ is $L_{\frac{x^m-1}{x^d-1} \cdot h} (\beta)$.
\end{proof}
In particular, we immediately get the following.
\begin{corollary}
	Let $\Fm/\F$ be a finite field extension. For every $b\in\F^*$, there exist exactly $\Phi(x^m-1)/(q-1)$ normal elements $\beta\in\Fm$, such that $\Tr(\beta)=b$.
\end{corollary}

The proof of the theorem below is inspired by the ideas found in the work of Reis~\cite{reis}.

\begin{theorem} \label{thm:normality}
	Let $m$ be an integer that is not a prime power and $1 < k < \sigma_0 (m)$, where $\sigma_0 (m)$ denotes the number of positive divisors of $m$. Let $\mathbf d = (d_1 , \ldots , d_k ) \in \lambda_k (m)$ and $\mathbf a = (a_1 , \ldots , a_k ) \in \mathbb F_{\mathbf d}$ be a normal $\mathbf d$-admissible $k$-tuple. 
	Set $g := \lcm (x^{d_1}-1 , \ldots , x^{d_k}-1)$. Then there
	exist exactly $\Phi(x^m-1)/\Phi(g)$ normal elements $\alpha\in\Fm$ with prescribed traces $\Tr_{m/d_i} (\alpha) = a_i$ for every $1 \leq i \leq k$.
\end{theorem}
\begin{proof}
	Fix some $\gamma\in\Fm$ normal over $\F$. For each $i=1,\ldots ,k$, set $c_i = \Tr_{m/d_i}(\gamma)$. From Lemma~\ref{lemma:normal_trace}, $c_i\in\mathbb{F}_{d_i}$ is normal, thus, there exists some $h_i$ (unique modulo $x^{d_i}-1$), relatively prime to $x^{d_i}-1$, such that $a_i = L_{h_i}(c_i)$. Furthermore, some $\alpha\in\Fm$ is normal if and only if $\alpha = L_F(\gamma)$, for some $F\in\F[x]$, that is relatively prime to $x^m-1$.
	
	It follows that, $\Tr_{m/d_i}(\alpha) = a_i$ if and only if $F\equiv h_i\pmod{x^{d_i}-1}$ for every $i=1,\ldots ,k$. Following the arguments from the proof of \cite[Theorem~4.1]{reis}, this congruence system has a unique solution modulo $g$, which we denote by $f$. Moreover, given that $\gcd(h_i,x^{d_i}-1)=1$ for all $i=1,\ldots ,k$, we readily obtain that $f+\langle g\rangle\in (\F[x]/\langle g\rangle)^*$. The desired result follows from the fact that Lemma~\ref{lem:aux_polys} entails that we have exactly $\Phi(x^m-1)/\Phi(g)$ choices for $F\in\F[x]$, that will be distinct modulo $x^m-1$ and relatively prime to $x^m-1$, such that $F\equiv f\pmod{x^m-1}$.
\end{proof}

\section{Intermediate traces of primitive normal elements} \label{sec:main_result}
Throughout this section, we adopt the same assumptions and notation as in Section~\ref{sec:normality}, with the additional assumption that $\mathbf a$ is normal.	Let $\N_{m,\mathbf{d},\mathbf{a}}$ 
be the number of primitive normal elements $\alpha \in \Fm$ with $\Tr_{m/d_i}(\alpha)=a_i$ for $i=1,\ldots ,k$. In particular,
\[ \N_{m,\mathbf{d},\mathbf{a}}
= \sum_{w \in \Fm}\rho_m(w)\cdot \kappa_m(w) \prod_{i=1}^{k}\tau_{m,d_i,a_i}(w) . \]

Since the $k$-tuple $(a_1,\ldots, a_k)$ is $\mathbf{d}$-admissible, we have seen that there exists some $\beta \in \Fm$ such that $\Tr_{(m/t)/{d_i}}(\beta) = a_i$ for $1\leq i \leq k$. Write $D = d_1 + \cdots + d_k$ and, for a generic $\mathbf{c} =
(c_1, \ldots, c_k) \in \mathbb{F}_{\mathbf{d}}$, write $s(\mathbf c) = \sum_{i=1}^{k}	c_i$. Now using the characteristic functions from Section~\ref{sec:prelim}, we get that
\begin{align*}
	\frac{q^D\cdot \N_{m,\mathbf{d},\mathbf{a}}}{\theta(q)\Theta(x^m-1)}= & \sum_{w \in \Fm}\sum_{\substack{t\mid q^m-1\\f\mid x^m-1}} \frac{\mu(t)\mu'(f)}{\phi(t)\Phi(f)}\sum_{\substack{\eta \in \Gamma(t)\\ \psi \in \Gamma(f)}} \eta(w)\psi(w)\cdot \prod_{i=1}^{k}\left(\sum_{c_i \in \mathbb{F}_{q^{d_i}}}\chi_{c_i}(w)\chi_{c_i}(\beta)^{-1}\right)\\
	= & \sum_{w \in \Fm}\sum_{\mathbf{c} \in \mathbb{F}_{\mathbf{d}}}\sum_{\substack{t\mid q^m-1\\f\mid x^m-1}} \frac{\mu(t)\mu'(f)}{\phi(t)\Phi(f)}\sum_{\substack{\eta \in \Gamma(t)\\ \psi \in \Gamma(f)}} \eta(w)\psi(w)\chi_{s(\mathbf{c})}(w)\chi_{s(\mathbf{c})}(-\beta)\\
	= & \sum_{w \in \Fm}\sum_{\mathbf{c} \in \mathbb{F}_{\mathbf{d}}}\sum_{\substack{t\mid q^m-1\\f\mid x^m-1}} \frac{\mu(t)\mu'(f)}{\phi(t)\Phi(f)}\sum_{\substack{\eta \in \Gamma(t)\\ \psi \in \Gamma(f)}} \eta(w)\chi_u(w)\chi_{s(\mathbf{c})}(w)\chi_{s(\mathbf{c})}(-\beta)\\
	= & \sum_{w \in \Fm}\sum_{\mathbf{c} \in \mathbb{F}_{\mathbf{d}}}\sum_{\substack{t\mid q^m-1\\f\mid x^m-1}} \frac{\mu(t)\mu'(f)}{\phi(t)\Phi(f)}\sum_{\substack{\eta \in \Gamma(t)\\ \psi \in \Gamma(f)}} \eta(w)\chi_{u+s(\mathbf{c})}(w)\chi_{s(\mathbf{c})}(-\beta)\\
	= & \sum_{\substack{t\mid q^m-1\\f\mid x^m-1}} \frac{\mu(t)\mu'(f)}{\phi(t)\Phi(f)}\sum_{\substack{\eta \in \Gamma(t)\\ \psi \in \Gamma(f)}}\sum_{\mathbf{c} \in \mathbb{F}_{\mathbf{d}}}\chi_{s(\mathbf{c})}(-\beta)G_m(\eta,\chi_{u+s(\mathbf{c})}),\\
\end{align*}
where $G_m(\eta, \chi_{u+s(\mathbf{c})})=\sum_{w \in \Fm} \eta(w)\cdot\chi_{u+s(\mathbf{c})}(w)$ denotes a Gauss sum.
%
In particular, we may rewrite
\[ \frac{q^D\cdot \N_{m,\mathbf{d},\mathbf{a}}}{\theta(q)\Theta(x^m-1)}= S_1+S_2, \]
where the term $S_1$ is the part of the above sum for $\eta \in \Gamma(1)$ 
and $S_2$ is the part for $\eta \notin \Gamma(1)$. 
Then $\theta(q)\Theta(x^m-1)S_1$ will denote the number of normal elements with their traces over $\mathbb F_{q^{d_i}}$ prescribed to $a_i$. Then, Theorem~\ref{thm:normality} yields
\[ S_1 = \frac{\Phi(x^m-1)}{\Phi(g)\theta(q)\Theta(x^m-1)} = \frac{q^m}{\Phi(g)\theta(q)} , \]
where $g=\lcm(x^{d_1}-1 ,\ldots ,x^{d_k}-1)$. Clearly, $\theta(q) \leq 1$ and $\Phi(g) < q^{\deg(g)} = q^{\lambda(\mathbf d)}$, hence,
\[ 
S_1 > q^{m-\lambda(\mathbf d)} .
\] 

Regarding $S_2$, we have that
\[ S_2= \sum_{\mathbf{c} \in \mathbb{F}_d}\sum_{\substack{t\mid q^m-1,t \neq 1\\f\mid x^m-1}} \frac{\mu(t)\mu'(f)}{\phi(t)\Phi(f)}\sum_{\substack{\eta \in \Gamma(t)\\ \chi_u \in \Gamma(f)}}\chi_{s(\mathbf{c})}(-\beta)G_m(\eta,\chi_{u+s(\mathbf{c})}) . \]
Recall that, for $\eta\notin \Gamma(1)$, the orthogonality relations and the well-known identity on Gauss sums yield that
\begin{enumerate}
	\item $G_m(\eta, \chi_{u+s(\mathbf{c})})= 0$, if $u+s(\mathbf{c})=0$, and
	\item $|G_m(\eta, \chi_{u+s(\mathbf{c})})|=q^{m/2}$, otherwise.
\end{enumerate}
Hence,
given that $|\chi_{s(\mathbf{c})}(-\beta)|=1$, 
that $|\Gamma(t)| = \phi(t)$, for all $t\mid q^m-1$, and that $|\Gamma(f)| = \Phi(f)$ for all $f\mid x^m-1$, we obtain
\[ |S_2| \leq 
q^{m/2+D}\cdot W(q^m-1) \cdot W(x^m-1). \]

Putting all of the above together,
\[ \frac{q^D\cdot \N_{m,\mathbf{d},\mathbf{a}}}{\theta(q)\Theta(q)} > q^{m-\lambda({\mathbf{d}})} - q^{m/2+D}\cdot W(q^m-1) \cdot W(x^m-1) . \]
Thus, $\N_{m,\mathbf{d},\mathbf{a}}>0$, provided that 
\[ q^{m/2-\lambda(\mathbf{d})-D} \geq W(q^m-1)\cdot W(x^m-1) . \]
Summarizing the above discussion, we have the following theorem.

\begin{theorem}\label{main}
	Let $m$ be an integer and $1<k<\sigma_0(m)$, where $\sigma_0(m)$ denotes the number of positive divisors of $m$. Let $\mathbf{d}= (d_1, \ldots , d_k) \in \lambda_k(m)$ and $\mathbf{a}=(a_1, \ldots, a_k) \in \mathbb{F}_{\mathbf{d}}$ be normal $\mathbf{d}$-admissible. Then there exists a primitive normal element $\alpha \in \Fm$ with prescribed traces $\Tr_{n/{d_i}}(\alpha)=a_i$ for every $1 \leq i \leq k$, provided that 
	\begin{equation}\label{3.1}
		q^{m/2-\lambda(\mathbf{d})-D} \geq W(q^m-1)\cdot W(x^m-1).
	\end{equation}
\end{theorem}

Furthermore, we have the following result which is an immediate consequence of \cite[Theorem~4.1]{reis} and the main theorem in \cite{cohen1999}. The idea of the proof is similar to that of \cite[Theorem~2.5]{ribas} and hence omitted.

\begin{theorem}\label{lcmcond}
	Keeping the notations as in Theorem~\ref{main}, we have that there exists a primitive normal element $\alpha \in \Fm$ with prescribed traces $\Tr_{n/{d_i}}(\alpha)=a_i$ for every $1 \leq i \leq k$ if $\lcm(d_1, \dots, d_k) < m$ holds.
\end{theorem}

\section{Existence results}\label{sec:existence_results}

In this section we explore the existence of primitive normal elements with prescribed traces in intermediate extensions and present explicit existence results. Although it is desirable to study the problem without any restrictions, due to the complexity of the expression of $\lambda(\mathbf{d})$ we restrict our study to the condition 
\[ \gcd(d_i, d_j)=1 ~\text{for} ~1 \leq i < j \leq k. \]
In particular, for $k \geq 2$, when $\gcd(d_i, d_j)=1$ for $1 \leq i < j \leq k$ we have that 
\[ 
\lambda(\mathbf{d}) = d_1 + \dots + d_k - k +1.
\] 

Next, let $p_i$ be the $i$-th prime. We have that $p_i \leq d_i$ and thus
\begin{equation}\label{cond5.2}
	p_t\leq d_t \leq \left( \frac{m}{p_1\cdots p_{t-1}}\right)^{1/(k+1-t)} ,
\end{equation}
for $1\leq t\leq k$, where the empty product equals 1.

Furthermore we may assume $\lcm(d_1, \dots, d_k) = m$, since otherwise by Theorem~\ref{lcmcond} we have the desired element. Also, since $\lcm(d_1, \dots, d_k) = m$ and $\gcd(d_i, d_j)=1$ for $1 \leq i < j \leq k$, we get that $d_1\cdots d_k=m$.

Under the above restrictions we obtain the following concrete and asymptotic results.

\begin{theorem}\label{th5.1}
	Let $m$ be an integer and $1<k<\sigma_0(m)$, where $\sigma_0(m)$ denotes the number of positive divisors of $m$. Let $\mathbf{d}= (d_1, \ldots , d_k) \in \lambda_k(m)$ and $\mathbf{a}=(a_1, \ldots, a_k) \in \mathbb{F}_{\mathbf{d}}$ be $\mathbf{d}$-admissible. Suppose $\gcd(d_i, d_j)=1$ for $1 \leq i < j \leq k$. Then there exists a primitive normal element $\alpha \in \Fm$ with prescribed traces $\Tr_{n/{d_i}}(\alpha)=a_i$ for every $1 \leq i \leq k$ provided that:
	
	\begin{enumerate}[label=(\roman*)]
		\item $k \geq 4$:
		\begin{enumerate}
			\item $k=4$ and $q \geq 1334$,
			\item $k=5$ and $q \geq 9$,
			\item $k=6$ and $q \geq 7$,
			\item $k=7$ and $q \geq 5$.
		\end{enumerate}
		\item $k = 3$, $m\geq 60$ and $q \geq 2.2660 \cdot 10^{24072855}$.
		\item $k=2$ and 
		\begin{enumerate}
			\item $d_1 \geq 8$ and $q$ large enough,
			\item $d_1 = 7$ and $q$ large enough with $(d_1, d_2) \neq (7, 8)$,
			\item $d_1 = 6$, $d_2\geq 13$ and $q$ large enough.
		\end{enumerate} 
	\end{enumerate}
\end{theorem}

\begin{proof}
	We split the proof into cases $k \geq 4$, $k=3$ and $k=2$. 
	
	We begin with the case $k \geq 4$.
	In this case, we have $m\geq 2\cdot3\cdot5\cdot7^{k-3}\geq 210$. Since the $d_j$ ’s are at least 2, pairwise relatively prime and $d_i\nmid d_j$ for every $1\leq i < j \leq k$, we have that $\frac{m}{d_i} = \prod_{j\neq i} d_j \geq \prod_{i=1}^{k-1} p_i$. Furthermore, since $\prod_{i=1}^{k-1} p_i\geq k!$ we have that $d_i \leq \lfloor\frac{m}{k!}\rfloor$. Therefore, $q^{\frac{m}{2}-\lambda(\mathbf{d})-D} \geq q^{\frac{m}{2}-2(d_1+d_2+\dots+d_k)+k-1} \geq q^{\frac{m}{2}-2k\lfloor\frac{m}{k!}\rfloor+k-1}$. Thus, combining the above with Ineq.~\eqref{3.1}, it suffices to verify that 
	\[ 
	q^{\frac{m}{2}-2k\lfloor\frac{m}{k!}\rfloor+k-1} \geq W(q^m-1)\cdot W(x^m-1).
	\] 
	The above, in conjunction with Lemmas~\ref{l2.4} and \ref{l2.5}, yield
	\[ 
	q^{\frac{m}{2}-2k\lfloor\frac{m}{k!}\rfloor+k-1} \geq \C q^{\frac{m}{\nu}}2^m.
	\] 
	In the case $k=4$, the above holds for $m\geq 210$, $q\geq 1334$, and $\nu=11$. For $k=5$, the above holds for $m\geq 2310$, $q\geq 9$, and $\nu=11$. Proceeding in the same way, for $k=6$ and $k=7$ the above inequality holds for $q\geq 7$ and $q\geq 5$ respectively for suitable values of $\nu$. Finally, we conclude this case by noting that for the cases $k\geq8$ the computations are challenging since the constants $\C$ for higher values of $\nu$ are difficult to calculate within a reasonable time limit. Furthermore, we note that for $k\geq8$ there is very less improvement to the lower bounds on $q$ and hence we stop at $k=7$ for which we have achieved the bound $q \geq5$. 
	
	We move on to the case $k = 3$.
	In this case, we have $m \geq 2\cdot3\cdot 5\geq 30$. Observe that, for $m=30=2\cdot3\cdot5$ and $m=42=2\cdot3\cdot7$, Ineq.~\eqref{3.1} does not hold for any prime power $q$. So, we focus on the case $m \neq 30, 42$, so we assume that $m\geq 60$. If $d_1=3$, $d_2=4$ and $d_3={m}/{12}$, then $d_1+d_2+d_3=3+4+\frac{m}{12}\leq \frac{m}{4}$ for $m \geq 60$.
	
	From Ineq.~\eqref{cond5.2} we have that $d_1 \leq \sqrt[3]{m}$ and $d_2\leq \sqrt{\frac{m}{2}}$. Now for $m\geq 70$, we get that $d_1+d_2+d_3\leq \sqrt[3]{m}+\sqrt{\frac{m}{2}}+\frac{m}{10}\leq \frac{m}{4}$. Thus, Ineq.~\eqref{3.1} yields the sufficient condition
	\[ 
	q^{2} \geq W(q^m-1)\cdot W(x^m-1).
	\] 
	Then, using Lemmas~\ref{l2.4} and \ref{l2.5}, the above inequality becomes 
	\[ 
	q^{2} \geq \C q^{\frac{m}{\nu}}2^m.
	\] 
	For $m\geq 60$, the above inequality is valid for $\nu> m/2=30$ and it holds for all prime powers $q \geq 2.2660 \cdot 10^{24072855}$ for $\nu=31$, given that $\mathcal{C}_{31}=2.4015 \cdot 10^{1553069}$.
	
	Finally, we focus on the case $k = 2$.
	We divide our discussion into the following cases.
	\begin{enumerate}
		\item For $8 \leq d_1 < d_2$, we have that $d_1 \leq m/8$ and $d_1+d_2\leq m/4$. Then from Theorem~\ref{3.1} it suffices to verify the inequality \[ q^{m/2-2(d_1+d_2)+1}\geq q\geq W(q^m-1)\cdot W(x^m-1). \] Then, Lemma~\ref{l2.6} ensures that we can compute a constant $Q'$ depending on $d_1$, $d_2$ and suitable values of $\nu$ such that Ineq.~\eqref{3.1} holds for all $q\geq Q'$
		\item For $7= d_1 < d_2$, we note that when $(d_1, d_2)=(7, 8)$, Ineq.~\eqref{3.1} does not hold for any prime power $q$.
		For the case $(d_1, d_2)=(7, 9)$, by considering Ineq.~\eqref{3.1}, it suffices to verify that \[ q^{1/2}\geq W(q^m-1)\cdot W(x^m-1) \] and the result follows as above.
		Finally, if $d_2 \geq 10$, we get that $d_2 \leq m/7$ and $d_1 \leq m/10$. In this case, Theorem~\ref{3.1} implies that it suffices to verify the inequality \[ q^{m/70+1}\geq W(q^m-1)\cdot W(x^m-1). \] Again, as above, Lemma~\ref{l2.6} ensures the existence of a computable constant $Q'$, depending on $d_1$, $d_2$ and suitable values of $\nu$, such that the Ineq.~\eqref{3.1} holds for all $q\geq Q'$.
		\item For $6= d_1 < d_2$, we note that when $7 \leq d_2 \leq 11$, Ineq.~\eqref{3.1} does not hold for any prime power $q$. Thus we work on $d_2 \geq 13$ and the result follows in a similar manner as above, from Lemma~\ref{l2.6}. 
		\item For $d_1 \in \{3, 4, 5\}$, Ineq.~\eqref{3.1} does not hold for any prime power $q$. 
	\end{enumerate}
	
\end{proof}

\subsection{Remarks}
\begin{enumerate}
	\item The condition $\gcd(d_i, d_j ) = 1$ is not restrictive for the case $k = 2$. In fact, if $\gcd(d_1, d_2) = d$ and $Q = q^d$, then $\mathbb{F}_{q^{d_i}} = \mathbb{F}_{q^{t_i}}$ where $t_i=d_i/d$ satisfies $\gcd(t_1, t_2)=1$. 
	\item Although we can explicitly compute the values of the constants $Q'$ above in the case $k=2$, these constants are so large as to prohibit the investigation of the situation for the prime powers smaller than $Q'$, by using a computer, within a reasonable time limit. So, we omit them.
	\item Recently, Bagger \cite{bagger1} provided a hybrid bound to attack problems on existence of primitive elements in finite fields. Furthermore, Bagger and Punch in \cite{bagger2}, provided a sieve criterion for primitive elements depending only on the estimate for a related character sum. We believe that an application of these methods adjusted, accordingly for the primitive normal elements, could be applied to this problem for future investigations. 
\end{enumerate}

\section{Acknowledgments}

We are grateful to the anonymous reviewer for their efforts in reviewing our manuscript and their suggestions which resulted in this improved version of the paper.

\section{Declarations}

The authors declare that there is no conflict of interest.
	

\end{document}